\documentclass[reqno, a4paper]{amsart}

\usepackage{amssymb, cancel, amsaddr, verbatim, graphicx, epstopdf}
\usepackage{amsmath}

\usepackage[scale=0.80, centering]{geometry}

\allowdisplaybreaks[4]

\usepackage{setspace}

\newcommand{\T}[1]{\tilde{#1}}

\newcommand{\vn}[1]{\lVert#1\rVert}

\newcommand{\IP}[2]{\left< #1 , #2 \right>}

\newcommand{\nDelta}{\Delta^\perp}
\newcommand{\nnabla}{\nabla^\perp}
\renewcommand{\H}{\vec{H}}

\newcommand{\R}{\ensuremath{\mathbb{R}}}

\newcommand{\SW}{\ensuremath{\mathcal{W{}}}}

\newcommand{\vSG}{\ensuremath{\vec{\mathcal{G{}}}}}
\newcommand{\vSW}{\ensuremath{\vec{\mathcal{W{}}}}}
\renewcommand{\T}{\ensuremath{\vec{T}}}

\newcommand{\SL}{\ensuremath{\mathcal{L{}}}}

\renewcommand{\H}{\ensuremath{\vec{H}}}

\newtheorem{thm}{Theorem}{\bf}{\it}
\newtheorem{cor}[thm]{Corollary}{\bf}{\it}
\newtheorem{prop}[thm]{Proposition}{\bf}{\it}
\newtheorem{lem}[thm]{Lemma}{\bf}{\it}
{\bf}{\it}

\theoremstyle{definition}
\newtheorem*{rmk}{Remark}{\bf}{\rm}

\begin{document}

\title
[Fourth-order gap phenomena for surfaces with boundary]
{Gap phenomena for a class of fourth-order geometric differential operators on surfaces with boundary}
\author{Glen Wheeler${}^{*}$}
\address{Otto-von-Guericke Universit\"at Magdeburg, Fakult\"at f\"ur Mathematik, Institut f\"ur
Analysis und Numerik, Universit\"atsplatz 2, 39106 Magdeburg, Germany\\[2mm]
\emph{Current:}
Institute for Mathematics and its Applications, University of Wollongong,
Northfields Ave, Wollongong, NSW 2522, Australia\\[2mm]
E-mail: {\tt glenw@uow.edu.au}
}
\thanks{${}^*$: Financial support from the Alexander-von-Humboldt Stiftung is gratefully
acknowledged}

\begin{abstract}
In this paper we establish a \emph{gap phenomenon} for immersed surfaces with
arbitrary codimension, topology and boundaries that satisfy one of a family
of systems of fourth-order anisotropic geometric partial differential equations.
Examples include Willmore surfaces, stationary solitons for the surface
diffusion flow, and biharmonic immersed surfaces in the sense of Chen.
On the boundary we enforce either \emph{umbilic} or \emph{flat} boundary
conditions: that the tracefree second fundamental form and its derivative or
the full second fundamental form and its derivative vanish.
For the umbilic boundary condition we prove that any surface with small
$L^2$-norm of the tracefree second fundamental form or full second fundamental form must be totally umbilic;
that is, a union of pieces of round spheres and flat planes. We prove that the stricter smallness condition allows consideration for a broader range of differential operators.
For the flat boundary condition we prove the same result with weaker hypotheses, allowing more general operators, and a stronger conclusion: only pieces of planes are allowed.
The method used relies only on the smallness assumption and thus holds without requiring the imposition of additional
symmetries.
The result holds in the class of surfaces with any genus and irrespective of the number or shape of the boundaries.
\end{abstract}

\maketitle

\section{Introduction}

Let us consider a complete isometric immersion $f:\Sigma\rightarrow\R^n$ of a smooth surface $\Sigma$ with
topological boundary $\partial\Sigma$.  We allow for $\partial\Sigma$ to be disconnected, empty, or non-smooth
throughout the paper. We do not make any assumptions a-priori on the topology of $\Sigma$ or of its
image $f(\Sigma)$.
Figures 1--3 illustrate some possibilities.

Suppose we are given a differential operator $\vSG$ which acts on $f$ to produce a system of partial
differential equations $\vSG(f)=0$.
Consider a tensor field $T$ on $\Sigma$.
We say that the operator $\vSG(f)$ gives rise to a \emph{gap phenomenon with respect to $T$ in $L^2$} if
the following holds: There exists a universal constant $\varepsilon>0$ such that any solution of
$\vSG(f) = 0$ with $\vn{T}^2_{L^2} \le \varepsilon$ must in fact satisfy $T \equiv 0$.

\begin{figure}[t]
\label{F1}
\includegraphics[width=5cm]{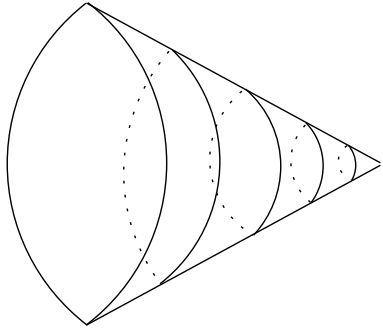}
\caption{The image $f(\Sigma)$ is depicted.
This piece of a cone has two disconnected components as its topological boundary, one
degenerate point-circle and another standard circle.}
\end{figure}

Gap phenomena are by now classical and prolific throughout the literature \cite{A90,F80,L69,R70,S68,Y74,Y09}.
In this paper we are concerned with identifying general conditions under which the operator $\vSG$ gives rise to gap phenomena with respect to the tracefree second fundamental form $A^o$ and the second fundamental form $A$.
We have chosen to concentrate on the case where $\vSG$ is anisotropic and fourth-order, not necessarily arising from a variational principle. On $\partial\Sigma$ we enforce that $|\nnabla A^o| = |A^o| = 0$, which we term
\emph{umbilic boundary conditions}.
The derivative $\nnabla$ is the induced connection on the normal bundle.
Our main results are gap phenomena with respect to $A$ and $A^o$ so long as $\vSG$ has $\nDelta\H$ as its leading order term and that the remaining nonlinearities may be bounded by an expression which is critical in the Sobolev
embedding sense.
Here we have used $\nDelta$ for the natural induced Laplacian in the normal bundle ($\Delta$ is the Laplace-Beltrami operator on $(\Sigma, f^*g^{\R^n})$) and $\H$ for the mean curvature vector corresponding to the immersion $f$.
One may think of the growth condition on the reaction terms in $\vSG$ as covering the subcritical and critical cases of second order differential operators acting on the curvature of $f$.  Certainly, one does not expect to find a
gap phenomenon in supercritical cases. This is not strictly true since a certain power of $|A^o|$ must be present for the gap phenomena to hold. Under other boundary conditions, such as the more restrictive \emph{flat boundary
conditions}, one can remove this restriction. This result is a perturbation of our main result (Theorem \ref{TMmt}) and we also state it.
Our precise assumptions and results are detailed in Section \ref{SCTsettingresults}.

Our primary motivating examples for the operator $\vSG$ are the Euler-Lagrange operator $\vSW(f) :=
\nDelta\H + A^o_{ij}\IP{(A^o)^{ij}}{\H}$ for the Willmore functional $\SW = \frac14\int_\Sigma
|\H|^2d\mu$, giving rise to Willmore surfaces, and the differential operators $\nDelta\nDelta$ and
$\Delta\nDelta$, giving rise stationary solitons for the surface diffusion flow and biharmonic
immersions respectively.

The paper is organised as follows.  In Section 2 we set our notation and give precise statements of our results.
In Section 3 we establish local estimates in $L^2$ for the differential operator $\vSG$ from below. There the umbilic boundary conditions are critical.
Section 4 is where we incorporate the various smallness conditions and prove Theorem \ref{TMmt}.
Here we require a version of the Michael-Simon Sobolev inequality for manifolds with boundary.
This is well-known, but a proof is difficult to find in the literature: for the convenience of the
reader, we have provided a proof in the appendix.


\section{Setting and main results}
\label{SCTsettingresults}

Suppose $\Sigma$ is a surface with boundary (or boundaries) $\partial\Sigma$ isometrically immersed
via a smooth immersion $f:\Sigma\rightarrow\R^n$, so that the Riemannian structure on $\Sigma$ is
given by $(\Sigma, f^*g^{\R^n})$ where $g^{\R^n}$ denotes the standard Euclidean metric on $\R^n$
and $f^*g^{\R^n}$ is the pullback metric or metric induced via $f$.
Consider the class of fourth-order geometric differential operators $\vSG$ which act on immersions
$f$ via
\begin{equation}
\label{EQgeneralpde}
\vSG(f) := a\nDelta\H + \T
\,,
\end{equation}
where $\H$ is the mean curvature vector of $f$, $\nDelta$ is the induced Laplacian on the normal
bundle $N\Sigma =  \big(T\Sigma)^\perp$, $a:\Sigma\rightarrow\R$ is a function and $\T$ is a section
of the normal bundle.  The function $a:\Sigma\rightarrow\R$ is assumed to be induced by an ambient
function $\tilde{a}:\R^n\rightarrow\R$ via the immersion $f$:
\begin{equation}
\label{EQanisotropypositive}
a(p) = (\tilde{a}\circ f)(p)\,\quad\text{ with }
\quad\inf_{x\in\R^n} \tilde{a}(x) = a_0 > 0\,.
\end{equation}
For the vector field $\T$, we assume that it is of the form $\T = \vec{\mathcal{T}}(f)$ where
$\vec{\mathcal{T}}$ is a second order differential operator with image in the normal bundle.
We additionally assume that $\T$ satisfies either the bound
\begin{equation}
\label{EQa0growthonT}
|\T|^2 \le c_0\big(|A|^2|A^o|^4 + |\nnabla A|^2|A^o|^2\big)
\end{equation}
or
\begin{equation}
\label{EQagrowthonT}
|\T|^2 \le c_1\big(|A|^{6-q}|A^o|^q + |\nnabla A|^2|A|^2\big)
\end{equation}
with $c_0$, $c_1$, $q\in(0,6]$ given constants.
In the above we have used $A$ to denote the second fundamental form of $f$, $A^o$ to denote the
tracefree second fundamental form, $\nabla^\perp$ to denote the induced covariant derivative in the
normal bundle, and $|\cdot|$ to denote the norm on tensor fields induced via the metric $g$.

\begin{figure}[t]
\label{F3}
\includegraphics[width=5cm]{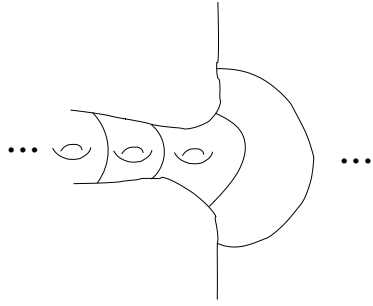}
\caption{The image $f(\Sigma)$ is depicted.
This surface has unbounded topological type.}
\end{figure}

The equation $\vSG(f) = 0$ when expressed in local coordinates is a strongly coupled system of fourth
order degenerate quasilinear partial differential equations.
We supplement \eqref{EQgeneralpde} with \emph{umbilic boundary conditions}, so called as they ensure
that $f$ is umbilic along $\partial\Sigma$: 
\begin{equation}
\label{EQbc}
|\nnabla A^o| = 
|A^o| = 0\,\qquad\text{on}\quad\partial\Sigma\,.
\end{equation}

We obtain a uniqueness theorem for solutions of $\vSG(f) \equiv 0$ satisfying \eqref{EQa0growthonT}
that are almost umbilic in a weak $L^2$-sense, as well as for solutions of $\vSG(f) \equiv 0$
satisfying \eqref{EQagrowthonT} that are almost flat in a weak $L^2$-sense.

\begin{thm}
\label{TMmt}
Suppose $\Sigma$ is an abstract two dimensional manifold with boundary (or boundaries) properly immersed via $f:\Sigma\rightarrow\R^n$.
Suppose $\vSG(f) = 0$ with umbilic boundary conditions
\eqref{EQbc} on $\partial\Sigma$.
\begin{enumerate}
\item If $\T$ satisfies \eqref{EQa0growthonT} then there exists an $\varepsilon > 0$ such that if
\begin{equation}
\label{EQsmallnessAo}
\int_\Sigma |A^o|^2d\mu < \varepsilon
\end{equation}
then $f$ is the union of pieces of round spheres and flat planes.

\item If $\T$ satisfies \eqref{EQagrowthonT} then there exists an $\varepsilon > 0$ such that if
\begin{equation}
\label{EQAsmallness}
\int_\Sigma |A|^2d\mu < \varepsilon
\end{equation}
then $f$ is the union of pieces of round spheres and flat planes.
\end{enumerate}
\end{thm}

We note that of course $\partial\Sigma = \emptyset$ is allowed.
It is also worthwhile to note that Theorem \ref{TMmt} applies to entire immersions and does not require any growth conditions at infinity.

If we impose \emph{flat boundary conditions},
\begin{equation}
\label{EQfbc}
|\nnabla A| = 
|A| = 0\,\qquad\text{on}\quad\partial\Sigma\,,
\end{equation}
then we may take $q=0$ in \eqref{EQagrowthonT}. The conclusion of the theorem is also strengthened as we only allow planes. A precise statement is:

\begin{thm}
\label{TMmt2}
Suppose $\Sigma$ is an abstract two dimensional manifold with boundary (or boundaries) properly immersed via $f:\Sigma\rightarrow\R^n$.
Suppose $\vSG(f) = 0$ with flat boundary conditions
\eqref{EQfbc} on $\partial\Sigma$.
\begin{enumerate}
\item If $\T$ satisfies \eqref{EQa0growthonT} then there exists an $\varepsilon > 0$ such that if
\begin{equation*}
\int_\Sigma |A^o|^2d\mu < \varepsilon
\end{equation*}
then $f$ is the union of pieces of flat planes.

\item If $\T$ satisfies \eqref{EQagrowthonT}, allowing also $q=0$, then there exists an $\varepsilon > 0$ such that if
\begin{equation*}
\int_\Sigma |A|^2d\mu < \varepsilon
\end{equation*}
then $f$ is the union of pieces of flat planes.
\end{enumerate}
\end{thm}

\begin{figure}[t]
\label{F2}
\includegraphics[width=5cm]{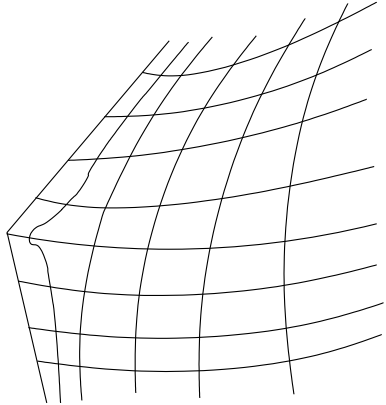}
\caption{The image $f(\Sigma)$ is depicted.
This surface with boundary is topologically a half-space. It is smooth, but its boundary has a corner. Here the corner is at the first order scale. Due to the boundary conditions \eqref{EQbc}, or \eqref{EQfbc}, any
corners must appear at the fourth or higher order scale.}
\end{figure}
\begin{rmk}
A priori, there are very few restrictions on the immersion $f$. It may possess arbitrary topology, boundaries, and so on. A posteriori, we know that $f$ is umbilic, being the union of pieces of planes and spheres. One may thus
conclude statements such as there being no toric immersion with $\vSG = 0$ satisfying the energy condition \eqref{EQsmallnessAo} for example ($\partial\Sigma \ne \emptyset$ is by no means required).
\end{rmk}

The proof of Theorem \ref{TMmt2} is almost identical to that of Theorem \ref{TMmt2}. There are two changes. First, one must set $q=0$ throughout and noting that the boundary term in \eqref{EQmsA1} now vanishes. Second, after the
conclusion that $f$ is umbilic, note that $|A|=0$ on $\partial\Sigma$ rules out spheres.

Many well-known differential operators are covered by theorems \ref{TMmt} and \ref{TMmt2} as special cases.  There are in particular three examples which we wish to enunciate: Willmore surfaces, stationary solitons of the surface
diffusion flow, and biharmonic surfaces.

{\it 1. Willmore surfaces.} An immersion $f$ is a Willmore surface if it satisfies
\begin{equation}
\label{EQwillmore}
\SW(f) := \nDelta \H + Q(A^o)\H = 0\quad\text{on }\Sigma\,,
\end{equation}
where $Q(A^o)$ acts on normal fields $\phi:\Sigma\rightarrow
N\Sigma$ by
\[
Q(A^o)\phi = A^o(e_i,e_j)\IP{A^o(e_i,e_j)}{\phi}\,,
\]
where $\{e_1,e_2\}$ is an orthonormal basis of the tangent bundle $T\Sigma$ of $\Sigma$ and the
Einstein summation convention is used.
The differential operator $\SW$ fits into the class considered here, with $\tilde{a}\equiv1$ and
satisfying \eqref{EQa0growthonT} with the estimate $|\T|^2 \le c|\H|^2|A^o|^4 \le c|A|^2|A^o|^4$.

Recently, Willmore surfaces with boundary have received quite a bit of attention.
For $\Sigma$ compact, Kuwert and Sch\"atzle proved the following gap phenomenon with
respect to $A^o$ for the Willmore operator.
\begin{thm}[{\cite[Theorem 2.7]{KS01}}]
There is an $\varepsilon>0$ such that any smooth solution of \eqref{EQwillmore} with
\[
\int_\Sigma |A^o|^2d\mu<\varepsilon \qquad\text{and}\qquad
\liminf_{\rho\rightarrow\infty}\int_{f^{-1}(B_\rho(0))}|A|^2d\mu = 0
\]
is a union of round spheres and flat planes.
\label{TMks}
\end{thm}
In \cite{MW12helclass} the growth condition at infinity was removed.
Here we further improve this by including the case of surfaces with boundary.
Uniqueness theorems for Willmore surfaces with boundary are also known; see Palmer \cite{palmer} and the recent extension by Dall'Acqua \cite{annauniqueness}.
There techniques inspired by Bryant's seminal work \cite{bryant} and the Pohozaev identity are used to obtain uniqueness through the use of symmetry without resorting to any small energy assumption.
For the results in \cite{annauniqueness} to hold, the shape of the boundary and the topology of the base manifold $\Sigma$ must be specified a priori.
Furthermore, as they critically use the classification of Bryant \cite{bryant}, they are restricted to the case of one codimension.

Although the boundary conditions considered here and in the works of Palmer and Dall'Acqua are different, theorems \ref{TMmt} and \ref{TMmt} may nevertheless be viewed as complementing these results in the sense that it confirms
one may trade in symmetry and topological assumptions on the boundary and on the surface itself, as well as the restriction to codimension one, for a smallness condition on the tracefree second fundamental form in $L^2$.

{\it 2. Stationary solitons for the surface diffusion flow.} The surface diffusion flow is the steepest descent $H^{-1}$-gradient flow for the area functional.  An immersed surface is a stationary soliton for the flow if
\begin{equation}
\label{EQsurfacediffusion}
\nDelta\H = 0\,.
\end{equation}
The differential operator \eqref{EQsurfacediffusion} is the simplest example of the class of operators \eqref{EQgeneralpde} which we study, with $\tilde{a} = 1$ and trivially satisfying \eqref{EQa0growthonT} with $|\T| = 0$.
A result analogous to \cite[Theorem 3.2]{KS01} for surface diffusion flow was established in one codimension in \cite{sdflowtospheres,mythesis}.
Theorem \ref{TMmt} of this paper generalises these results to arbitrary codimension and to the case of surfaces with boundary.

The case of anisotropic surface diffusion flow has recently received some attention \cite{D12sd,DG11sd}, where one studies the steepest descent $H^{-1}$-gradient flow of the functional $\int_\Sigma a\,d\mu$.
Stationary solitons for this flow are also covered by our theorems.

{\it 3. Biharmonic surfaces.} An immersed surface is termed \emph{biharmonic} or \emph{biharmonic in the sense of Chen} if
\begin{equation}
\label{EQbiharmonic}
\Delta\H = 0
\end{equation}
where $\Delta$ is the metric (or rough) Laplacian.
The normal component of this equation must vanish if \eqref{EQbiharmonic} is satisfied; that is, 
\begin{equation}
\label{EQbiharmonicNormal}
\nDelta\H = H^\beta A_{ij}^\alpha A_{ij}^\beta \nu^\alpha\,.
\end{equation}
Classification questions for biharmonic surfaces have a rich history, motivated primarily by the
study of Chen's conjecture \cite{Ch91}, which claims that all biharmonic submanifolds of
Euclidean space are minimal.
It is easy to check that while round spheres clearly satisfy both \eqref{EQwillmore} and
\eqref{EQsurfacediffusion}, they do \emph{not} satisfy \eqref{EQbiharmonicNormal}.
They certainly satisfy the condition \eqref{EQagrowthonT} with $q=0$, and so Theorem \ref{TMmt2} gives uniqueness for biharmonic surfaces with $A$ small in $L^2$.

\begin{rmk}
Chen's conjecture for biharmonic submanifolds claims that submanifolds of $\R^n$ with $\Delta\H = 0$ are harmonic.
Here we have proven that surfaces immersed in $\R^n$ satisfying $\Delta\H = 0$ with flat boundary and $\vn{A}_2^2 < \varepsilon$ are \emph{flat}.
This is the first progress on Chen's conjecture for surfaces with boundary.
\end{rmk}


\section{Local estimates for $\vSG$ from below in $L^2$}


The localisation we shall use is the function 
\begin{equation}
\label{EQgamma}
\gamma(p) = (\tilde\gamma\circ f)(p),\quad \gamma(p)\in[0,1]\,,
\end{equation}
for $p\in\Sigma$ and $\tilde\gamma$ a $C^1$ function on $\R^n$ with compact
support. Suppose $||\nabla\gamma||_\infty \le \frac{1}\rho$ for some $\rho$
depending on $\tilde\gamma$ to be set later.

\begin{lem}
\label{LMlm1}
Suppose $f:\Sigma\rightarrow\R^n$ is an immersed surface with boundaries satisfying \eqref{EQbc}, and $\gamma$ is a function as in \eqref{EQgamma}.
Then
\begin{align*}
\int_\Sigma &\big(|\nnabla_{(2)}\H|^2 + |\H|^2|\nnabla\H|^2\big)\gamma^4\,d\mu
\\
&\le c\int_\Sigma |\nDelta\H|^2\gamma^4\,d\mu
   + c\int_\Sigma |A^o|^2|\nnabla\H|^2\gamma^4\,d\mu
   + \frac{c}{\rho^2}\int_\Sigma |\nnabla A^o|^2\gamma^2d\mu
\,,
\end{align*}
where $c$ depends only on $n$ and $\rho$ depends only on $\tilde\gamma$.
\end{lem}
\begin{proof}
Interchange of covariant derivatives and the Codazzi equation gives the standard formula
\begin{equation}
\label{EQ1}
\nnabla\nDelta\H = \nDelta\nnabla\H - \frac14 |\H|^2\nnabla\H + A*A^o*\nnabla\H\,.
\end{equation}
In the equation above we have denoted by $*$ contraction with the metric $g$ and possible multiplication by a constant.
Integrating \eqref{EQ1} against $\nnabla \H\,\gamma^4$ yields
\begin{align}
\int_\Sigma \IP{\nnabla\nDelta \H}{\nnabla \H}_g \gamma^4 d\mu
 &=
    \int_\Sigma \IP{\nDelta\nnabla \H}{\nnabla \H}_g \gamma^4 d\mu
\notag
\\
\label{EQ2}
&\hskip-2cm
  - \frac14 \int_\Sigma |\H|^2|\nnabla \H|^2\gamma^4d\mu
  + \int_\Sigma A*A^o*\nnabla \H*\nnabla \H\, \gamma^4d\mu.
\end{align}
Using the divergence theorem we have
\begin{align}
\int_\Sigma \IP{\nnabla\nDelta \H}{\nnabla \H}_g \gamma^4 d\mu
 &= -\int_\Sigma |\nDelta \H|^2\gamma^4d\mu - 4\int_\Sigma \nDelta \H \IP{\nnabla \H}{\nabla\gamma}_g\gamma^3d\mu
\notag
\\&\quad + \int_{\partial\Sigma} (\nDelta \H)\cdot(\nnabla_\nu \H) \gamma^4 d\mu\,,
\label{EQ3}
\end{align}
where $\nu$ is the outward normal to $\partial\Sigma$.
From the Codazzi equation it follows that
\begin{equation}
\label{EQ4}
\nnabla_j\H = 2\nnabla_i (A^o)^i_j := 2(\nnabla \ast A^o)_j\,,
\end{equation}
where we have slightly abused notation and used $\ast$ as shorthand for the specific metric divergence operation above.
Since $\nabla A^o = 0$ on $\partial\Sigma$, we have that $\nabla_\nu\H = 0$ on $\partial\Sigma$ and so the boundary term
in \eqref{EQ3} vanishes.
Applying the divergence theorem once more we find
\begin{align}
\int_\Sigma \IP{\nDelta\nnabla \H}{\nnabla \H}_g \gamma^4 d\mu
 &= -\int_\Sigma |\nnabla_{(2)}\H|^2\gamma^4d\mu - 4\int_\Sigma \IP{\nnabla_{(2)} \H}{\nabla\gamma\nnabla \H}_g\gamma^3d\mu
\notag
\\&\quad + \frac12\int_{\partial\Sigma} \nnabla_\nu|\nnabla \H|^2 \gamma^4 d\mu^\partial\,.
\label{EQ5}
\end{align}
Note that the boundary term again vanishes due to equation \eqref{EQ4}.
Combining \eqref{EQ3}, \eqref{EQ5} with \eqref{EQ2} we obtain
\begin{align*}
  \int_\Sigma |\nnabla_{(2)}\H|^2\gamma^4d\mu
&+ \frac14 \int_\Sigma |\H|^2|\nnabla \H|^2\gamma^4d\mu
=
    \int_\Sigma |\nDelta \H|^2\gamma^4d\mu
\\
&
  + \int_\Sigma \nnabla_{(2)} \H*\nnabla \H*\nabla \gamma\ \gamma^3d\mu
  + \int_\Sigma A*A^o*\nnabla \H*\nnabla \H\ \gamma^4d\mu
\\
&\le
    \int_\Sigma |\nDelta \H|^2\gamma^4d\mu
  + \frac12\int_\Sigma |\nnabla_{(2)}\H|^2\gamma^4d\mu
\\
&\quad
  + \frac{c}{\rho^2}\int_\Sigma |\nnabla A^o|^2 \gamma^2d\mu
  + \int_\Sigma (A^o+g\H)*A^o*\nnabla \H*\nnabla \H\ \gamma^4d\mu
\\
&\le
    \int_\Sigma |\nDelta \H|^2\gamma^4d\mu
  + \frac12\int_\Sigma |\nnabla_{(2)}\H|^2\gamma^4d\mu
  + \frac18\int_\Sigma |\H|^2|\nnabla \H|^2\gamma^4d\mu
\\
&\quad
  + \frac{c}{\rho^2}\int_\Sigma |\nnabla A^o|^2 \gamma^2d\mu
  + c\int_\Sigma |A^o|^2|\nnabla \H|^2 \gamma^4d\mu\,.
\end{align*}
Absorbing finishes the proof.
\end{proof}

\begin{lem}
\label{LMlm2}
Suppose $f:\Sigma\rightarrow\R^n$ is an immersed surface with boundaries satisfying \eqref{EQbc}, and $\gamma$ is a function as in \eqref{EQgamma}.
Then
\begin{align*}
\int_\Sigma &\big(|\H|^4|A^o|^2 + |\H|^2|\nnabla A^o|^2\big)\gamma^4\,d\mu
\\
&\le c\int_\Sigma |\H|^2|\nnabla\H|^2\gamma^4\,d\mu
   + c\int_\Sigma |A^o|^2|\nnabla A^o|^2\gamma^4\,d\mu
   + c\int_\Sigma |A^o|^6\gamma^4\,d\mu
  + \frac{c}{\rho^4}\int_{[\gamma>0]}|A^o|^2d\mu
\,,
\end{align*}
where $c$ depends only on $n$ and $\rho$ depends only on $\tilde\gamma$.
\end{lem}
\begin{proof}
Simons' identity implies
\begin{equation}
\label{EQ6}
\nDelta A^o = S^o(\nnabla_{(2)}\H) + \frac12|\H|^2A^o + A^o*A^o*A^o\,,
\end{equation}
where $S^o(B)$ denotes the tracefree part of the symmetric bilinear form $B$.
Integrating \eqref{EQ6} against $|\H|^2A^o$ we obtain
\begin{align*}
\int_\Sigma &|\H|^2|\nnabla A^o|^2\gamma^4d\mu
- \int_{\partial\Sigma} |\H|^2\IP{\nnabla_\nu A^o}{A^o}\gamma^4d\mu^\partial
\\
&= - \int_\Sigma |\H|^2\IP{A^o}{\nDelta A^o}_g\gamma^4d\mu
\\
&\qquad
   - 2\int_\Sigma \IP{\nnabla A^o}{\H\cdot\nnabla \H\, A^o}_g\gamma^4d\mu
   - 4\int_\Sigma |\H|^2\IP{\nnabla A^o}{\nabla\gamma\, A^o}_g\gamma^3d\mu
\\
&= - \int_\Sigma |\H|^2\IP{A^o}{S^o(\nnabla_{(2)}\H) + \frac12|\H|^2A^o + A^o*A^o*A^o}_g\gamma^4d\mu
\\
&\qquad
   - 2\int_\Sigma \IP{\nnabla A^o}{\H\cdot\nnabla H\, A^o}_g\gamma^4d\mu
   - 4\int_\Sigma |\H|^2\IP{\nnabla A^o}{\nabla\gamma\, A^o}_g\gamma^3d\mu
\\
&= - \int_\Sigma |\H|^2\IP{A^o}{\nnabla_{(2)}\H}_g\gamma^4d\mu
   - \frac12\int_\Sigma |\H|^4|A^o|^2\gamma^4d\mu
   + \int_\Sigma |\H|^2A^o*A^o*A^o*A^o\,\gamma^4d\mu
\\
&\qquad
   - 2\int_\Sigma \IP{\nnabla A^o}{\H\cdot\nnabla H\, A^o}_g\gamma^4d\mu
   - 4\int_\Sigma |\H|^2\IP{\nnabla A^o}{\nabla\gamma\, A^o}_g\gamma^3d\mu
\\
&=  \int_\Sigma |\H|^2\IP{\nnabla\ast A^o}{\nnabla\H}_g\gamma^4d\mu
  - \int_{\partial\Sigma} |\H|^2\IP{A^o(\nu,\cdot)}{\nnabla\H}\gamma^4d\mu^\partial
  + 2\int_\Sigma \IP{A^o}{(\H\cdot\nnabla\H)\nnabla\H}_g\gamma^4d\mu
\\
&\qquad
   - \frac12\int_\Sigma |\H|^4|A^o|^2\gamma^4d\mu
   + \int_\Sigma |\H|^2A^o*A^o*A^o*A^o\gamma^4d\mu
   - 2\int_\Sigma \IP{\nnabla A^o}{\H\cdot\nnabla H\, A^o}_g\gamma^4d\mu
\\
&\qquad
   - 4\int_\Sigma |\H|^2\IP{\nnabla A^o}{\nabla\gamma\, A^o}_g\gamma^3d\mu
   + 4\int_\Sigma |\H|^2\IP{A^o}{\nabla\gamma\nnabla \H}_g\gamma^3d\mu\,.
\end{align*}
Noting \eqref{EQbc} and \eqref{EQ4} we estimate the right hand side to obtain for $\delta > 0$
\begin{align*}
\int_\Sigma &|\H|^2|\nnabla A^o|^2\gamma^4d\mu
   + \frac12\int_\Sigma |\H|^4|A^o|^2\gamma^4d\mu
\\
&=  \frac12\int_\Sigma |\H|^2|\nnabla \H|^2\gamma^4d\mu
  + 2\int_\Sigma \IP{A^o}{(\H\cdot\nnabla\H)\nnabla\H}_g\gamma^4d\mu
\\
&\qquad
   + \int_\Sigma |\H|^2A^o*A^o*A^o*A^o\gamma^4d\mu
   - 2\int_\Sigma \IP{\nnabla A^o}{\H\cdot\nnabla H\, A^o}_g\gamma^4d\mu
\\
&\qquad
   - 4\int_\Sigma |\H|^2\IP{\nnabla A^o}{\nabla\gamma\, A^o}_g\gamma^3d\mu
   + 4\int_\Sigma |\H|^2\IP{A^o}{\nabla\gamma\nnabla \H}_g\gamma^3d\mu
\\
&\le \Big(\frac12+\delta\Big)\int_\Sigma |\H|^2|\nnabla \H|^2\gamma^4d\mu
  + \frac{c}{\delta}\int_\Sigma |A^o|^2|\nnabla \H|^2\gamma^4d\mu
\\
&\qquad
  + \delta\int_\Sigma |\H|^4|A^o|^2\gamma^4d\mu
  + \frac{c}{\delta}\int_\Sigma \big(|A^o|^6 + |A^o|^2|\nnabla A^o|^2\big)\gamma^4d\mu
\\
&\qquad
  + \delta\int_\Sigma |\H|^2|\nnabla A^o|^2\gamma^4d\mu
  + \frac{c}{\rho^2\delta}\int_\Sigma |\H|^2|A^o|^2\gamma^2d\mu
\\
&\le \Big(\frac12+\delta\Big)\int_\Sigma |\H|^2|\nnabla \H|^2\gamma^4d\mu
  + \delta\int_\Sigma |\H|^2|\nnabla A^o|^2\gamma^4d\mu
\\
&\qquad
  + \delta\int_\Sigma |\H|^4|A^o|^2\gamma^4d\mu
  + \frac{c}{\delta}\int_\Sigma \big(|A^o|^6 + |A^o|^2|\nnabla A^o|^2\big)\gamma^4d\mu
\\
&\qquad
  + \frac{c}{\rho^4\delta^3}\int_{[\gamma>0]}|A^o|^2d\mu\,.
\end{align*}
Absorbing the second and third terms from the right hand side into the left finishes the proof.
\end{proof}

\begin{cor}
\label{CY1}
Suppose $f:\Sigma\rightarrow\R^n$ is an immersed surface with boundaries satisfying \eqref{EQbc}, and $\gamma$ is a function as in \eqref{EQgamma}.
Then
\begin{align*}
\int_\Sigma &\big(|A|^4|A^o|^2 + |A|^2|\nnabla A|^2\big)\gamma^4\,d\mu
\\
&\le c\int_\Sigma |\H|^2|\nnabla\H|^2\gamma^4\,d\mu
   + c\int_\Sigma |A^o|^2|\nnabla A^o|^2\gamma^4\,d\mu
   + c\int_\Sigma |A^o|^6\gamma^4\,d\mu
  + \frac{c}{\rho^4}\int_{[\gamma>0]}|A^o|^2d\mu
\,,
\end{align*}
where $c$ depends only on $n$ and $\rho$ depends only on $\tilde\gamma$.
\end{cor}
\begin{proof}
Codazzi implies
\begin{equation}
\label{EQ7}
|\nnabla_{(k)} A| \le c|\nnabla_{(k)} A^o|
\end{equation}
for any integer $k\ge1$. Noting also the decomposition $A = A^o + \frac12 g\H$ we compute
\begin{align*}
|\H|^4|A^o|^2 + |\H|^2|\nnabla A^o|^2
 &= |2A - 2A^o|^4|A^o|^2 + |2A-2A^o|^2|\nnabla A^o|^2
\\&= 16(|A|^2-2\IP{A}{A^o}_g + |A^o|^2)^2|A^o|^2 + 4(|A|^2-2\IP{A}{A^o}_g + |A^o|^2)|\nnabla A^o|^2
\\&\ge 8(|A|^2 - 2|A^o|^2)^2|A^o|^2 + 2(|A|^2 - 2|A^o|^2)|\nnabla A^o|^2
\\&\ge 8(|A|^4 - 4|A|^2|A^o|^2 + 4|A^o|^4)|A^o|^2 + 2(|A|^2 - 2|A^o|^2)|\nnabla A^o|^2
\\&\ge 4(|A|^4 - 8|A^o|^4)|A^o|^2 + 2(|A|^2 - 2|A^o|^2)|\nnabla A^o|^2\,.
\end{align*}
Summarising,
\begin{equation}
\label{EQ8}
4|A|^4|A^o|^2 + 2|A|^2|\nnabla A^o|^2
\le |\H|^4|A^o|^2 + |\H|^2|\nnabla A^o|^2
  + 32|A^o|^6 + 4|A^o|^2|\nnabla A^o|^2\,.
\end{equation}
Combining \eqref{EQ8} with Lemma \ref{LMlm2} and \eqref{EQ7} we find
\begin{align*}
4\int_\Sigma |A|^4|A^o|^2 \gamma^4d\mu + 2\int_\Sigma |A|^2|\nnabla A|^2 \gamma^4d\mu
&\le
4\int_\Sigma |A|^4|A^o|^2 \gamma^4d\mu + c\int_\Sigma |A|^2|\nnabla A^o|^2 \gamma^4d\mu
\\
&\le
 c\int_\Sigma \big(|\H|^4|A^o|^2 + |\H|^2|\nnabla A^o|^2\big)\gamma^4\,d\mu
\\&\quad
  + c\int_\Sigma |A^o|^6 \gamma^4d\mu + 4\int_\Sigma |A^o|^2|\nnabla A^o|^2 \gamma^4 d\mu
\\
&\le c\int_\Sigma |\H|^2|\nnabla\H|^2\gamma^4\,d\mu
   + c\int_\Sigma |A^o|^2|\nnabla A^o|^2\gamma^4\,d\mu
\\&\quad
   + c\int_\Sigma |A^o|^6\gamma^4\,d\mu
  + \frac{c}{\rho^4}\int_{[\gamma>0]}|A^o|^2d\mu\,.
\end{align*}

\end{proof}

Our argument now proceeds in two directions: one, for operators $\vSG$ with reaction terms satisfying \eqref{EQa0growthonT}, the other for operators
$\vSG$ with reaction terms satisfying \eqref{EQagrowthonT}.

\begin{lem}
\label{LMlm3}
Suppose $f:\Sigma\rightarrow\R^n$ is an immersed surface with boundaries satisfying \eqref{EQbc}, and $\gamma$ is a function as in \eqref{EQgamma}.
Assume the operator $\vSG$ is of the form \eqref{EQgeneralpde} satisfying \eqref{EQanisotropypositive} and \eqref{EQa0growthonT}.
Then
\begin{align*}
\int_\Sigma &\big(|\nnabla_{(2)}\H|^2 + |A|^2|\nnabla A|^2 + |A|^4|A^o|^2\big)\gamma^4\,d\mu
\\
&\le c\int_\Sigma |\vSG|^2\gamma^4\,d\mu
   + c\int_\Sigma \big(|A^o|^6 + |\nnabla A^o|^2|A^o|^2\big)\gamma^4\,d\mu
   + \frac{c}{\rho^2}\int_\Sigma |\nnabla A^o|^2\gamma^2d\mu
  + \frac{c}{\rho^4}\int_{[\gamma>0]}|A^o|^2d\mu
\,,
\end{align*}
where $c$ depends only on $n$, $a_0$, $c_0$, and $\rho$ depends only on $\tilde\gamma$.
\end{lem}
\begin{proof}
Combining Lemma \ref{LMlm1} with Corollary \ref{CY1} we first obtain
\begin{align}
\int_\Sigma &\big(|\nnabla_{(2)}\H|^2 + |A|^2|\nnabla A|^2+|A|^4|A^o|^2\big)\gamma^4\,d\mu
\notag\\
&\le c\int_\Sigma |\nDelta\H|^2\gamma^4\,d\mu
   + c\int_\Sigma \big(|A^o|^6 + |\nnabla A^o|^2|A^o|^2\big)\gamma^4\,d\mu
   + \frac{c}{\rho^2}\int_\Sigma |\nnabla A^o|^2\gamma^2d\mu
   + \frac{c}{\rho^4}\int_{[\gamma>0]}|A^o|^2d\mu\,.
\label{EQ9}
\end{align}
Now
\begin{align}
\nDelta\H &= a^{-1}\vSG - a^{-1}\T
\intertext{so using \eqref{EQanisotropypositive} and \eqref{EQa0growthonT}}
|\nDelta\H|^2 &\le 2a_0^{-1}|\vSG|^2 + 2a_0^{-1}|\T|^2
\notag\\
&\le 2a_0^{-1}|\vSG|^2 + 2ca_0^{-1} \Big(|A|^2|A^o|^4 + |\nnabla A|^2|A^o|^2\Big)
\notag\\
&\le 2a_0^{-1}|\vSG|^2 + 2ca_0^{-1} \Big(\varepsilon|A|^4|A^o|^2 + \frac{1}{4\varepsilon}|A^o|^6 + c|\nnabla A^o|^2|A^o|^2\Big)
\,.
\label{EQ10}
\end{align}
Choosing $\varepsilon$ small enough and absorbing by combining \eqref{EQ9} with \eqref{EQ10} finishes the proof.
\end{proof}

\begin{lem}
\label{LMlm4}
Suppose $f:\Sigma\rightarrow\R^n$ is an immersed surface with boundaries satisfying \eqref{EQbc}, and $\gamma$ is a function as in \eqref{EQgamma}.
Assume the operator $\vSG$ is of the form \eqref{EQgeneralpde} satisfying \eqref{EQanisotropypositive} and \eqref{EQagrowthonT}.
Then
\begin{align*}
\int_\Sigma |\nnabla_{(2)}\H|^2\gamma^4\,d\mu
\le c\int_\Sigma |\vSG|^2\gamma^4\,d\mu
   + c\int_\Sigma \big(|A|^{6-q}|A^o|^q + |\nnabla A|^2|A|^2\big)\gamma^4\,d\mu
   + \frac{c}{\rho^2}\int_\Sigma |\nnabla A|^2\gamma^2d\mu
  + \frac{c}{\rho^4}\int_{[\gamma>0]}|A|^2d\mu
\,,\
\end{align*}
where $c$ depends only on $n$, $a_0$, $c_1$, and $\rho$ depends only on $\tilde\gamma$.
\end{lem}
\begin{proof}
The proof is as above for Lemma \ref{LMlm3} except the estimate \eqref{EQ10} is modified to
\begin{align}
|\nDelta\H|^2 &\le 2a_0^{-1}|\vSG|^2 + 2a_0^{-1}|\T|^2
\notag\\
&\le 2a_0^{-1}|\vSG|^2 + 2ca_0^{-1} \Big(|A|^{6-q}|A^o|^q + |\nnabla A|^2|A|^2\Big)
\,.
\notag
\end{align}
Due to the weaker condition on $\T$ this is the best we can obtain.
In this case we throw away the terms on the left hand side of the form $|A|^4|A^o|^2 + |\nnabla A|^2|A^o|^2$ since they will be of no use.
Estimating $|\nnabla A^o| \le |\nnabla A|$ and $|A^o| \le |A|$ then proceeding again as in the proof of Lemma \ref{LMlm3} finishes the proof.
\end{proof}

\begin{lem}
\label{LMlm5}
Suppose $f:\Sigma\rightarrow\R^n$ is an immersed surface with boundaries satisfying \eqref{EQbc}, and $\gamma$ is a function as in \eqref{EQgamma}.
Assume the operator $\vSG$ is of the form \eqref{EQgeneralpde} satisfying \eqref{EQanisotropypositive} and \eqref{EQa0growthonT}.
Then
\begin{align*}
\int_\Sigma &\big(|\nnabla_{(2)} A|^2 + |A|^2|\nnabla A|^2 + |A|^4|A^o|^2\big)\gamma^4\,d\mu
\le c\int_\Sigma |\vSG|^2\gamma^4\,d\mu
   + c\int_\Sigma \big(|A^o|^6 + |\nnabla A^o|^2|A^o|^2\big)\gamma^4\,d\mu
  + \frac{c}{\rho^4}\int_{[\gamma>0]}|A^o|^2d\mu
\,.
\end{align*}
\end{lem}
\begin{proof}
Let us first note that \eqref{EQ6} allows us to estimate
\begin{equation*}
|\nDelta A^o|^2 \le c|\nnabla_{(2)}\H|^2 + c|\H|^4|A^o|^2 + c|A^o|^6\,.
\end{equation*}
Since $c$ above is absolute, we obtain for some small $\varepsilon > 0$ that
\begin{equation}
\label{EQ11}
\varepsilon|\nDelta A^o|^2 \le |\nnabla_{(2)}\H|^2 + |A|^4|A^o|^2 + |A^o|^6\,,
\end{equation}
and so by combining \eqref{EQ11} with Lemma \ref{LMlm3} we obtain the improvement
\begin{align}
\notag
\int_\Sigma &\big(|\nDelta A^o|^2 + |A|^2|\nnabla A|^2 + |A|^4|A^o|^2\big)\gamma^4\,d\mu
\\
&\le c\int_\Sigma |\vSG|^2\gamma^4\,d\mu
   + c\int_\Sigma \big(|A^o|^6 + |\nnabla A^o|^2|A^o|^2\big)\gamma^4\,d\mu
   + \frac{c}{\rho^2}\int_\Sigma |\nnabla A^o|^2\gamma^2d\mu
  + \frac{c}{\rho^4}\int_{[\gamma>0]}|A^o|^2d\mu
\,.
\label{EQ11.5}
\end{align}
We may now deal with the third integral on the right hand side.
The divergence theorem and \eqref{EQbc} gives
\begin{align*}
   \frac{c'}{\rho^2}\int_\Sigma |\nnabla A^o|^2\gamma^2d\mu
&=
   -\frac{c'}{\rho^2}\int_\Sigma \IP{A^o}{\nDelta A^o}_g^2\gamma^2d\mu
   -\frac{c'}{\rho^2}\int_\Sigma \IP{\nabla\gamma\,A^o}{\nnabla A^o}_g^2\gamma d\mu
   +\frac{c'}{\rho^2}\int_{\partial\Sigma} \IP{A^o}{\nnabla_\nu A^o}_g^2\gamma^2d\mu^\partial
\\&\le 
    \frac14\int_\Sigma |\nDelta A^o|^2\gamma^4\,d\mu
  + \frac{c'}{2\rho^2}\int_\Sigma |\nnabla A^o|^2\gamma^2d\mu
  + \frac{c}{\rho^4}\int_{[\gamma>0]}|A^o|^2d\mu
\end{align*}
so that we obtain
\begin{equation}
\label{EQ12}
   \frac{c'}{\rho^2}\int_\Sigma |\nnabla A^o|^2\gamma^2d\mu
\le 
    \frac12\int_\Sigma |\nDelta A^o|^2\gamma^4\,d\mu
  + \frac{c}{\rho^4}\int_{[\gamma>0]}|A^o|^2d\mu\,.
\end{equation}
Combining \eqref{EQ12} with \eqref{EQ11.5} above gives
\begin{align}
\int_\Sigma &\big(|\nDelta A^o|^2 + |A|^2|\nnabla A|^2 + |A|^4|A^o|^2\big)\gamma^4\,d\mu
\\
&\le c\int_\Sigma |\vSG|^2\gamma^4\,d\mu
   + c\int_\Sigma \big(|A^o|^6 + |\nnabla A^o|^2|A^o|^2\big)\gamma^4\,d\mu
  + \frac{c}{\rho^4}\int_{[\gamma>0]}|A^o|^2d\mu
\,.
\label{EQ12.5}
\end{align}
Let us now prove
\begin{align}
\int_\Sigma |\nnabla_{(2)}A^o|^2\gamma^4d\mu
 &\le
    \int_\Sigma |\nDelta A^o|^2\gamma^4d\mu 
  + c\int_\Sigma |A|^2|\nnabla A|^2\gamma^4d\mu
\notag
\\
&\hskip+2cm
  + \frac{c}{\rho^2}\int_\Sigma |\nnabla A^o|^2\gamma^2d\mu
\,.
\label{EQ13}
\end{align}
We begin by using a consequence of the interchange formula for covariant derivatives:
\begin{equation}
\label{EQ13.5}
\nDelta \nnabla A^o = \nnabla \nDelta A^o + \nnabla A^o * A * A .
\end{equation}
Testing \eqref{EQ13.5} against $\nnabla A^o\,\gamma^4$ and using \eqref{EQbc} with the divergence theorem we find
\begin{align}
\int_\Sigma |\nnabla_{(2)}A^o|^2\gamma^4d\mu
 &=
 - \int_\Sigma \IP{\nnabla A^o}{\nDelta \nnabla A^o}_g\gamma^4d\mu
 - 4\int_\Sigma \IP{\nabla\gamma\nnabla A^o}{\nnabla_{(2)}A^o}_g \gamma^3d\mu
\notag\\&\quad
 + \int_{\partial\Sigma} \IP{\nnabla_\nu\nnabla A^o}{\nnabla A^o}\gamma^4d\mu^\partial
\notag
\\
 &=
 - \int_\Sigma \IP{\nnabla A^o}{\nnabla\Delta  A^o}_g\gamma^4d\mu
 - 4\int_\Sigma \IP{\nabla\gamma\nnabla A^o}{\nnabla_{(2)}A^o}_g \gamma^3d\mu
\notag
\\
 &\quad + \int_\Sigma \nnabla A^o * \nnabla A^o * A * A\ \gamma^4d\mu
\notag
\\
 &=
  \int_\Sigma |\nDelta  A^o|^2\gamma^4d\mu
 + 4\int_\Sigma \IP{\nnabla A^o}{\nabla\gamma\nDelta A^o}_g \gamma^3d\mu
\notag
\\
 &\quad
 - \int_{\partial\Sigma} \IP{\nnabla_\nu A^o}{\nDelta A^o}\gamma^4d\mu^\partial
\notag
\\
 &\quad
 - 4\int_\Sigma \IP{\nabla\gamma\nnabla A^o}{\nnabla_{(2)}A^o}_g \gamma^3d\mu
 + \int_\Sigma \nnabla A^o * \nnabla A^o * A * A\ \gamma^4d\mu
\notag
\\
 &\le
    \int_\Sigma |\nDelta A^o|^2\gamma^4d\mu 
  + c\int_\Sigma |A|^2|\nnabla A^o|^2\gamma^4d\mu
\notag
\\
 &\quad
  + \int_\Sigma \nnabla_{(2)}A^o*\nnabla A^o*\nabla\gamma\ \gamma^3d\mu
\notag
\\
 &\le
    \int_\Sigma |\nDelta A^o|^2\gamma^4d\mu 
  + c\int_\Sigma |A|^2|\nnabla A^o|^2\gamma^4d\mu
\notag
\\
 &\quad
  + \frac{c}{\rho^2}\int_\Sigma |\nnabla A^o|^2 \gamma^2d\mu
  + \frac12\int_\Sigma |\nnabla_{(2)}A^o|^2\gamma^4d\mu
\notag
\end{align}
which by absorption implies \eqref{EQ13}.
Now from \eqref{EQ12} we improve \eqref{EQ13} to
\begin{align}
\int_\Sigma |\nnabla_{(2)}A^o|^2\gamma^4d\mu
 &\le
    \int_\Sigma |\nDelta A^o|^2\gamma^4d\mu 
  + c\int_\Sigma |A|^2|\nnabla A|^2\gamma^4d\mu
  + \frac{c}{\rho^4}\int_{[\gamma>0]} |A^o|^4d\mu
\,.
\label{EQ14}
\end{align}
As before with \eqref{EQ12} we multiply \eqref{EQ14} by a small constant and combine with \eqref{EQ12.5} to find
\begin{align}
\int_\Sigma &\big(|\nnabla_{(2)} A^o|^2 + |A|^2|\nnabla A|^2 + |A|^4|A^o|^2\big)\gamma^4\,d\mu
\notag\\
&\le c\int_\Sigma |\vSG|^2\gamma^4\,d\mu
   + c\int_\Sigma \big(|A^o|^6 + |\nnabla A^o|^2|A^o|^2\big)\gamma^4\,d\mu
  + \frac{c}{\rho^4}\int_{[\gamma>0]}|A^o|^2d\mu
\,.
\label{EQ15}
\end{align}
Estimating the leading order term in \eqref{EQ15} above from below using \eqref{EQ7} finishes the proof.
\end{proof}

\begin{lem}
\label{LMlm6}
Suppose $f:\Sigma\rightarrow\R^n$ is an immersed surface with boundaries satisfying \eqref{EQbc}, and $\gamma$ is a function as in \eqref{EQgamma}.
Assume the operator $\vSG$ is of the form \eqref{EQgeneralpde} satisfying \eqref{EQanisotropypositive} and \eqref{EQagrowthonT}.
Then
\begin{align*}
\int_\Sigma |\nnabla_{(2)}A|^2\gamma^4\,d\mu
\le c\int_\Sigma |\vSG|^2\gamma^4\,d\mu
   + c\int_\Sigma \big(|A|^{6-q}|A^o|^q + |\nnabla A|^2|A|^2\big)\gamma^4\,d\mu
   + \frac{c}{\rho^4}\int_{[\gamma>0]}|A|^2d\mu
\,,\
\end{align*}
where $c$ depends only on $n$, $a_0$, $c_1$, and $\rho$ depends only on $\tilde\gamma$.
\end{lem}
\begin{proof}
The proof is as for Lemma \ref{LMlm5} above except we again throw away the terms on the left hand side of the form
$|A|^4|A^o|^2 + |\nnabla A|^2|A^o|^2$ since they will be of no use.
The third integral on the right hand side from Lemma \ref{LMlm4} is dealt with by an estimate analogous to \eqref{EQ12}.
\end{proof}


\section{Almost umbilic and almost flat in $L^2$}

We shall combine the estimates from Section 3 with the smallness assumptions and the Michael-Simon
Sobolev inequality \cite{MS73} for manifolds with boundary:

\begin{thm}
\label{MSSbdy}
Suppose $f:M^m\rightarrow\R^n$ is a smooth immersion of the $m$-dimensional manifold $M$ with boundary $\partial M$ into $\R^n$. Then for any $u\in C^1(\overline{M})$
\[
\Big(\int_M |u|^{m/(m-1}\,d\mu\Big)^{(m-1)/2}
\le \frac{4^{m+1}}{\omega_m^{1/m}} \Big(
    \int_M |\nabla u| + |\H||u|\,d\mu + \int_{\partial M} |u|\,d\mu^\partial
                                   \Big)\,.
\]
\end{thm}

The proof of Theorem \ref{MSSbdy} is a straightforward application of the standard Michael-Simon
Sobolev ineqaulity and is a well-known folklore result.  As the details are difficult to find in the
literature, we have provided a proof in the appendix for the convenience of the reader.

We now need boundary versions of the multiplicative Sobolev inequalities from \cite{KS01,KS02}.

\begin{lem}
\label{LMmsAo}
Suppose $f:\Sigma\rightarrow\R^n$ is an immersed surface with boundaries satisfying \eqref{EQbc}, and $\gamma$ is a function as in \eqref{EQgamma}.
Then
\begin{align*}
\int_\Sigma |A^o|^6\gamma^4d\mu
&+ \int_\Sigma |\nnabla A^o|^2|A^o|^2\gamma^4d\mu
\\*
 &\le c\vn{A^o}^2_{2,[\gamma>0]}\int_\Sigma \big(|\nnabla_{(2)} A|^2 + |\nnabla A|^2|A|^2 + |A|^2|A^o|^4\big)\,\gamma^4d\mu
    + c\rho^{-4}\vn{A^o}^4_{2,[\gamma>0]}
\,,
\end{align*}
where $c$ depends only on $n$, and $\rho$ depends only on $\tilde\gamma$.
\end{lem}
\begin{proof}
Applying Theorem \ref{MSSbdy} with $u = |A^o|^3\gamma^2$, estimating and using \eqref{EQbc} we find
\begin{align}
\int_\Sigma |A^o|^6\gamma^4d\mu
 &\le c\bigg(\int_\Sigma \big(|A^o|^2|\nnabla A^o|\,\gamma^2 + \rho^{-1}|A^o|^3\gamma + |\H|\,|A^o|^3\gamma^2\big)\,d\mu\bigg)^2
    + c\bigg(\int_{\partial\Sigma} |A^o|^3\gamma^2d\mu^\partial\bigg)^2
\notag\\
 &\le c\vn{A^o}^2_{2,[\gamma>0]}\int_\Sigma \big(|\nnabla A^o|^2|A^o|^2 + |A|^2|A^o|^4\big)\,\gamma^4d\mu
    + c\rho^{-4}\vn{A^o}^4_{2,[\gamma>0]}
\label{EQmsAo1}
\,,
\end{align}
where we used
\begin{align*}
 c\rho^{-2}\bigg(\int_\Sigma |A^o|^3\gamma\,d\mu\bigg)^2
&\le 
    c\rho^{-4}\vn{A^o}^4_{2,[\gamma>0]}
  + c\bigg(\int_\Sigma |A^o|^4\gamma^2\,d\mu\bigg)^2
\\&\le 
    c\rho^{-4}\vn{A^o}^4_{2,[\gamma>0]}
  + c\rho^{-2}\vn{A^o}^2_{2,[\gamma>0]}\int_\Sigma |A^o|^6\gamma^4d\mu
\\&\le 
    c\rho^{-4}\vn{A^o}^4_{2,[\gamma>0]}
  + c\rho^{-2}\vn{A^o}^2_{2,[\gamma>0]}\int_\Sigma |A|^2|A^o|^4\gamma^4d\mu
\,.
\end{align*}
Now we apply Theorem \ref{MSSbdy} with $u = |\nnabla A^o|\,|A^o|\,\gamma^2$ and use \eqref{EQbc} again to obtain
\begin{align}
\int_\Sigma |\nnabla A^o|^2|A^o|^2\gamma^4d\mu
 &\le c\bigg(\int_\Sigma \big(|A^o|\,|\nnabla_{(2)} A^o|\,\gamma^2 + |\nnabla A^o|^2\gamma^2
\notag\\&\qquad + \rho^{-1}|\nnabla A^o|\,|A^o|\,\gamma^2 + |\H|\,|A^o|\,|\nnabla A^o|\,\gamma^2\big)\,d\mu\bigg)^2
    + c\bigg(\int_{\partial\Sigma} |A^o|\,|\nnabla A^o|\,\gamma^2d\mu^\partial\bigg)^2
\notag\\
 &\le c\vn{A^o}^2_{2,[\gamma>0]}\int_\Sigma \big(|\nnabla_{(2)} A^o|^2 + |\nnabla A^o|^2|A|^2 \big)\,\gamma^4d\mu
\notag\\&\qquad
    + c\bigg(\int_\Sigma |\nnabla A^o|^2\gamma^2d\mu\bigg)^2 + c\rho^{-4}\vn{A^o}^4_{2,[\gamma>0]}
\label{EQmsAo2}
\,.
\end{align}
Let us remove the second integral on the right hand side of \eqref{EQmsAo2}.
First estimate
\begin{align*}
\bigg(\int_\Sigma |\nnabla A^o|^2\gamma^2d\mu\bigg)^2
 &=
\bigg(
- \int_\Sigma \IP{\nDelta A^o}{A^o}\gamma^2d\mu
- 2\int_\Sigma \IP{\nnabla A^o}{\nabla\gamma\,A^o}\gamma\,d\mu
+ \int_{\partial\Sigma} \IP{\nnabla_\nu A^o}{A^o}_g\gamma^2d\mu
\bigg)^2
\\
 &\le
      c\vn{A^o}^2_{2,[\gamma>0]}\int_\Sigma \big(|\nnabla_{(2)} A^o|^2 + |\nnabla A^o|^2|A|^2 \big)\,\gamma^4d\mu
    + c\rho^{-2}\vn{A^o}^2_{2,[\gamma>0]}\int_\Sigma |\nnabla_{(2)} A^o|^2 \gamma^4d\mu
\\
 &\le
      c\vn{A^o}^2_{2,[\gamma>0]}\int_\Sigma \big(|\nnabla_{(2)} A^o|^2 + |\nnabla A^o|^2|A|^2 \big)\,\gamma^4d\mu
\\&\qquad
    + \frac12\bigg(\int_\Sigma |\nnabla A^o|^2\gamma^2d\mu\bigg)^2
    + c\rho^{-4}\vn{A^o}^4_{2,[\gamma>0]}
\end{align*}
so that absorbing gives
\begin{equation}
\bigg(\int_\Sigma |\nnabla A^o|^2\gamma^2d\mu\bigg)^2
 \le
      c\vn{A^o}^2_{2,[\gamma>0]}\int_\Sigma \big(|\nnabla_{(2)} A^o|^2 + |\nnabla A^o|^2|A|^2 \big)\,\gamma^4d\mu
    + c\rho^{-4}\vn{A^o}^4_{2,[\gamma>0]}
\,.
\label{EQmsAo2.5}
\end{equation}
Combining \eqref{EQmsAo2.5} with \eqref{EQmsAo2} gives
\begin{align}
\int_\Sigma |\nnabla A^o|^2|A^o|^2\gamma^4d\mu
 &\le c\vn{A^o}^2_{2,[\gamma>0]}\int_\Sigma \big(|\nnabla_{(2)} A^o|^2 + |\nnabla A^o|^2|A|^2 \big)\,\gamma^4d\mu
    + c\rho^{-4}\vn{A^o}^4_{2,[\gamma>0]}
\label{EQmsAo3}
\,.
\end{align}
Noting \eqref{EQ7} while adding together \eqref{EQmsAo1} and \eqref{EQmsAo3} gives the result.
\end{proof}

\begin{lem}
\label{LMmsA}
Suppose $f:\Sigma\rightarrow\R^n$ is an immersed surface with boundaries satisfying \eqref{EQbc}, and $\gamma$ is a function as in \eqref{EQgamma}.
Let $q\in(0,6]$. Then
\begin{align*}
\int_\Sigma |A^o|^q|A|^{6-q}\gamma^4d\mu
&+ \int_\Sigma |\nnabla A|^2|A|^2\gamma^4d\mu
\\*
 &\le c\vn{A}^2_{2,[\gamma>0]}\int_\Sigma \big(|\nnabla_{(2)} A|^2 + |\nnabla A|^2|A|^2 + |A^o|^q|A|^{6-q}\big)\,\gamma^4d\mu
    + c\rho^{-4}\vn{A}^4_{2,[\gamma>0]}
\,,
\end{align*}
where $c$ depends only on $n$, $q$, and $\rho$ depends only on $\tilde\gamma$.
\end{lem}
\begin{proof}
As this is similar to the proof of Lemma \ref{LMmsAo} above, we present the proof only briefly.
The tricky part is the application of the boundary conditions \eqref{EQbc}. At the zero order level, we do not have $|A| = 0$ on $\partial\Sigma$ and this causes the tracefree second fundamental form to make a compulsory
appearance in the growth condition for $\vSG$.
It is also not in general true that $|\nnabla A| = 0$ on $\partial\Sigma$.  Indeed, the Codazzi relations on $\partial\Sigma$ are not enough to obtain this and so one must be careful with the boundary terms arising from the
divergence theorem.

We begin by applying Theorem \ref{MSSbdy} with $u = |A^o|^p|A|^{3-p}\,\gamma^2$ where $2p = q$ to find
\begin{align}
\int_\Sigma |A^o|^{2p}|A|^{6-2p}\gamma^4d\mu
 &\le c\bigg(\int_\Sigma \big(|A|^2|\nnabla A|\,\gamma^2 + \rho^{-1}|A^o|^p|A|^{3-p}\gamma + |A^o|^p|A|^{4-p}\gamma^2\big)\,d\mu\bigg)^2
\notag\\&\qquad
    + c\bigg(\int_{\partial\Sigma} |A^o|^p|A|^{3-p}\gamma^2d\mu^\partial\bigg)^2
\notag\\
 &\le c\vn{A}^2_{2,[\gamma>0]}\int_\Sigma \big(|\nnabla A|^2|A|^2 + |A|^{6-2p}|A^o|^{2p}\big)\,\gamma^4d\mu
    + c\rho^{-4}\vn{A}^4_{2,[\gamma>0]}
\label{EQmsA1}
\,.
\end{align}
Now we apply \eqref{EQ7} and Theorem \ref{MSSbdy} with $u = |\nnabla A^o|\,|A|\,\gamma^2$ to obtain
\begin{align}
\int_\Sigma &|\nnabla A|^2|A|^2\gamma^4d\mu
\le \int_\Sigma |\nnabla A^o|^2|A|^2\gamma^4d\mu
\notag\\&
 \le c\bigg(\int_\Sigma \big(|A|\,|\nnabla_{(2)} A^o|\,\gamma^2 + |\nnabla A^o|^2\gamma^2
    + \rho^{-1}|\nnabla A^o|\,|A|\,\gamma^2 + |\H|\,|A|\,|\nnabla A^o|\,\gamma^2\big)\,d\mu\bigg)^2
\notag\\&\qquad
    + c\bigg(\int_{\partial\Sigma} |A|\,|\nnabla A^o|\,\gamma^2d\mu^\partial\bigg)^2
\notag\\
 &\le c\vn{A^o}^2_{2,[\gamma>0]}\int_\Sigma \big(|\nnabla_{(2)} A|^2 + |\nnabla A|^2|A|^2 \big)\,\gamma^4d\mu
    + c\bigg(\int_\Sigma |\nnabla A^o|^2\gamma^2d\mu\bigg)^2 + c\rho^{-4}\vn{A}^4_{2,[\gamma>0]}
\notag\\
 &\le c\vn{A^o}^2_{2,[\gamma>0]}\int_\Sigma \big(|\nnabla_{(2)} A|^2 + |\nnabla A|^2|A|^2 \big)\,\gamma^4d\mu
    + c\rho^{-4}\vn{A}^4_{2,[\gamma>0]}
\label{EQmsA2}
\,,
\end{align}
where we used an argument analogous to \eqref{EQmsAo2.5}.
Adding together \eqref{EQmsA1} and \eqref{EQmsA2} finishes the proof.
\end{proof}

\begin{prop}
\label{PPfinal}
Suppose $f:\Sigma\rightarrow\R^n$ is an immersed surface with boundaries satisfying \eqref{EQbc}.
Assume the operator $\vSG$ is of the form \eqref{EQgeneralpde} satisfying \eqref{EQanisotropypositive} and \eqref{EQa0growthonT}.
Then there exists a universal $\varepsilon>0$ such that if
\[
\int_\Sigma |A^o|^2d\mu \le \varepsilon
\]
then
\begin{align*}
\int_\Sigma &\big(|\nnabla_{(2)}A|^2 + |\nnabla A|^2|A|^2 + |A|^4|A^o|^2\big)\gamma^4\,d\mu
\le c\int_\Sigma |\vSG|^2\gamma^4\,d\mu + \frac{c}{\rho^4}\int_{[\gamma^4>0]} |A^o|^2d\mu
\,.
\end{align*}
\end{prop}
\begin{proof}
Combine Lemma \ref{LMlm5} with Lemma \ref{LMmsAo} and absorb.
\end{proof}

\begin{prop}
\label{PPfinal2}
Suppose $f:\Sigma\rightarrow\R^n$ is an immersed surface with boundaries satisfying \eqref{EQbc} and
$\T$ satisfying \eqref{EQagrowthonT}.
Then there exists a universal $\varepsilon>0$ such that if
\[
\int_\Sigma |A|^2d\mu \le \varepsilon
\]
then
\begin{align*}
\int_\Sigma &\big(|\nnabla_{(2)}A|^2 + |\nnabla A|^2|A|^2 + |A^o|^q|A|^{6-q}\big)\gamma^4\,d\mu
\le c\int_\Sigma |\vSG|^2\eta\,d\mu
   + \frac{c}{\rho^4}\int_{[\eta>0]} |A|^2d\mu
\,.
\end{align*}
\end{prop}
\begin{proof}
Combine Lemma \ref{LMlm6} with Lemma \ref{LMmsA}, add the integral
\[
\int_\Sigma \big(|A^o|^q|A|^{6-q} + |\nnabla A|^2|A|^2\big)\gamma^4d\mu
\]
to both sides, and absorb.
\end{proof}

\begin{proof}[Proof of Theorem \ref{TMmt}]
Choose $\tilde\gamma$ to be a cutoff function on an ambient ball of radius $r>0$. We may guarantee that
\[
\rho = cr
\]
for some constant $c$ depending only on $n$. Using $\vSG = 0$ and taking $r\rightarrow\infty$ (recall that all boundary integrals are calculated above to vanish)
in each of Propositions \ref{PPfinal} and \ref{PPfinal2}, we find $|A^o| = 0$ and so $f$ is an umbilic.

To obtain the full statement in case \eqref{EQAsmallness}, just note that in order for pieces of spheres to be possible each boundary $\partial\Sigma$ must be compact and must lie on the surface of a 2-sphere sitting in $\R^n$.
\end{proof}

\begin{rmk}
Since $\vn{A}_2^2$ is scale invariant, the only way to satisfy \eqref{EQAsmallness} is by each boundary being pulled sufficiently tight. Suppose $P$ is a piece of a sphere with radius $\rho$. Then clearly
\[
\int_{f^{-1}(P)}|A|^2d\mu < 4\pi^2\,.
\]
A positive lower bound is not possible, but it is possible to find a lower bound in terms of the diameter of $\partial P$ in $\R^n$. This would sharpen the statement of Theorem \ref{TMmt}.
\end{rmk}


\appendix

\section*{Appendix}

\begin{proof}[Proof of Theorem \ref{MSSbdy}]
Let $\Lambda:M\rightarrow(0,\infty)$ be the distance to $\partial M$ and consider the family of functions $\sigma_k(p) = \text{min}\{u(p),k\Lambda(p)u(p)\}$ for $p\in M$.
Let us now approximate the family $\sigma_k$ in the $C^1$ topology; this approximation we also denote by $\sigma_k$. We have
\begin{align*}
|[u \ne \sigma_k]| \longrightarrow 0\quad
\\
\Big(\int_M |\sigma_k|^{m/(m-1}\,d\mu\Big)^{(m-1)/2}
 \longrightarrow
\Big(\int_M |u|^{m/(m-1}\,d\mu\Big)^{(m-1)/2}
\\
\int_M |\H||\sigma_k|\,d\mu
 \longrightarrow
\int_M |\H||u|\,d\mu
\end{align*}
as $k\rightarrow\infty$.
We clearly also have $\vn{\nabla \sigma_k}_1 \le \vn{|\nabla u|\text{min}\{1,k\Lambda\}}_1 + \vn{u\nabla\text{min}\{1,k\Lambda\}}_1$.
Using Fermi coordinates on a neighbourhood of $\partial M$ we find for large $k$
\begin{align*}
\int_M |u|\,|\nabla\text{min}\{1,k\Lambda\}|\,d\mu
\le
k\int_0^{\frac1k}\int_{\partial M} |u|\sqrt{\text{det}(g_{ij}(p,t)}\,d\SL^{m-1}\,dt
\rightarrow
\int_{\partial M} |u|\,d\mu^\partial
\end{align*}
as $k\rightarrow\infty$.
Applying the standard Michael-Simon Sobolev inequality \cite{MS73} to the family $\sigma_k$ and taking $k\rightarrow\infty$ finishes the proof.
\end{proof}

\section*{Acknowledgements}

Part of this work was completed while the author was an Alexander von Humboldt research fellow at the Otto-von-Guericke Universit\"at Magdeburg. 
The author is grateful for the support of the Alexander von Humboldt Stiftung.
Part of this work was completed while the author was a guest of the group of Harald Garcke at Regensburg.  He is grateful for their hospitality.
Part of this work was completed while the author was a research associate at the Institute for Mathematics and its Applications at the University of Wollongong, on the Australian Research Council's Discovery Projects scheme
(project number DP120100097).

\bibliographystyle{plain}
\bibliography{critical}

\end{document}